\documentclass[amsart,11pt]{article}
\usepackage{comment}
\usepackage{color,graphicx,latexsym,amssymb,amsmath,amsfonts,amsthm,float}
\usepackage{url}
\usepackage{authblk,hyperref}
\usepackage{caption}
\usepackage{subcaption}

\newtheorem{theorem}{Theorem}
\newtheorem{lemma}[theorem]{Lemma}

\newtheorem{conj}{Conjecture}

\theoremstyle{definition}

\newtheorem{definition}{Definition}

\definecolor{db}{rgb}{0.1,0,0.75}
\definecolor{lm}{cmyk}{0 ,1,0,0}

\newcommand{\Z}{\mathbb Z}
\newcommand{\R}{\mathbb R}

\newcommand{\old}[1]{}
\newcommand{\dd}{\delta}

\newtheorem{rem}{Remark}
\newcommand{\CD}{\textbf{Frozen-Boundary Diffusion}}
\newcommand{\CC}{\textbf{FBD}}

\title{Heat diffusion with frozen boundary}
\author[1]{Laura Florescu\thanks{florescu@cims.nyu.edu}}
\author[2]{Shirshendu Ganguly\thanks{sganguly@math.washington.edu}}
\author[3]{Yuval Peres\thanks{peres@microsoft.com}}
\author[4]{Joel Spencer\thanks{spencer@cims.nyu.edu}}
\affil[1]{New York University}
\affil[2]{University of Washington}
\affil[3]{Microsoft Research} 
\affil[4]{New York University}

\begin{document}
\maketitle

\begin{abstract}
Consider ``Frozen Random Walk" on $\mathbb{Z}$: $n$ particles start at the origin. At any discrete time, the leftmost and rightmost  $\lfloor{\frac{n}{4}}\rfloor$ particles are ``frozen" and do not move. The rest of the particles in the ``bulk" independently jump to the left and right uniformly. The goal of this note is to understand the limit of this process under scaling of mass and time.
To this end we study the following deterministic mass splitting process: start with mass $1$ at the origin. At each step the extreme quarter mass on each side is ``frozen". The remaining ``free" mass in the center evolves according to the discrete heat equation. We establish diffusive behavior of this mass evolution and identify the scaling limit under the assumption of its existence. It is natural to expect the limit to be a truncated Gaussian. A naive guess for the truncation point might be the $1/4$ quantile points on either side of the origin. We show that this is not the case and it is in fact determined by the evolution of the second moment of the mass distribution. 

%
\end{abstract}

\section{Introduction}\label{def1}
The goal of this note is to understand the long term behavior of the mass evolution process which is a divisible version of the particle system ``Frozen Random Walk". 
We define \textbf{Frozen-Boundary Diffusion} with parameter $\alpha$ (or \CC-$\alpha$) as follows. Informally it is a sequence $\mu_t$ of symmetric probability distributions on $\mathbb{Z}.$ The sequence has the following recursive definition: given $\mu_t$, the leftmost and rightmost $\frac{\alpha}{2}$ masses are constrained to not move, and the remaining $1-\alpha$ mass diffuses according to one step of the discrete heat equation to yield $\mu_{t+1}$. In other words, we split the mass at site $x$ equally to its two neighbors. Formal descriptions appear later. We briefly remark that this process is similar to Stefan type problems, which have been studied for example in \cite{gravner2000}.

Now we also introduce the random counterpart of \CC-$\alpha$. We define the frozen random walk process (Frozen Random Walk-$(n,1/2)$) as follows: $n$ particles start at the origin. At any discrete time the leftmost and rightmost  $\lfloor{n\frac{\alpha}{2}}\rfloor$ particles are ``frozen" and do not move. The remaining $n-2\lfloor{n\frac{\alpha}{2}}\rfloor$ particles  independently jump to the left and right uniformly.
Letting $n \to \infty$ and fixing $t$, the mass distribution for the above random process converges to the $t^{th}$ element, $\mu_t$, in \CC-$\alpha$. However, if $t$ and $n$ simultaneously go to $\infty$, one has to control the fluctuations to be able to prove any limiting statement.
Figure \ref{f.mass} depicts the mass distribution $\mu_{t}$ and the frozen random walk process for $\alpha=\frac{1}{2}$.

\begin{center}
\begin{figure}[H]
\centering
\includegraphics[scale=0.55]{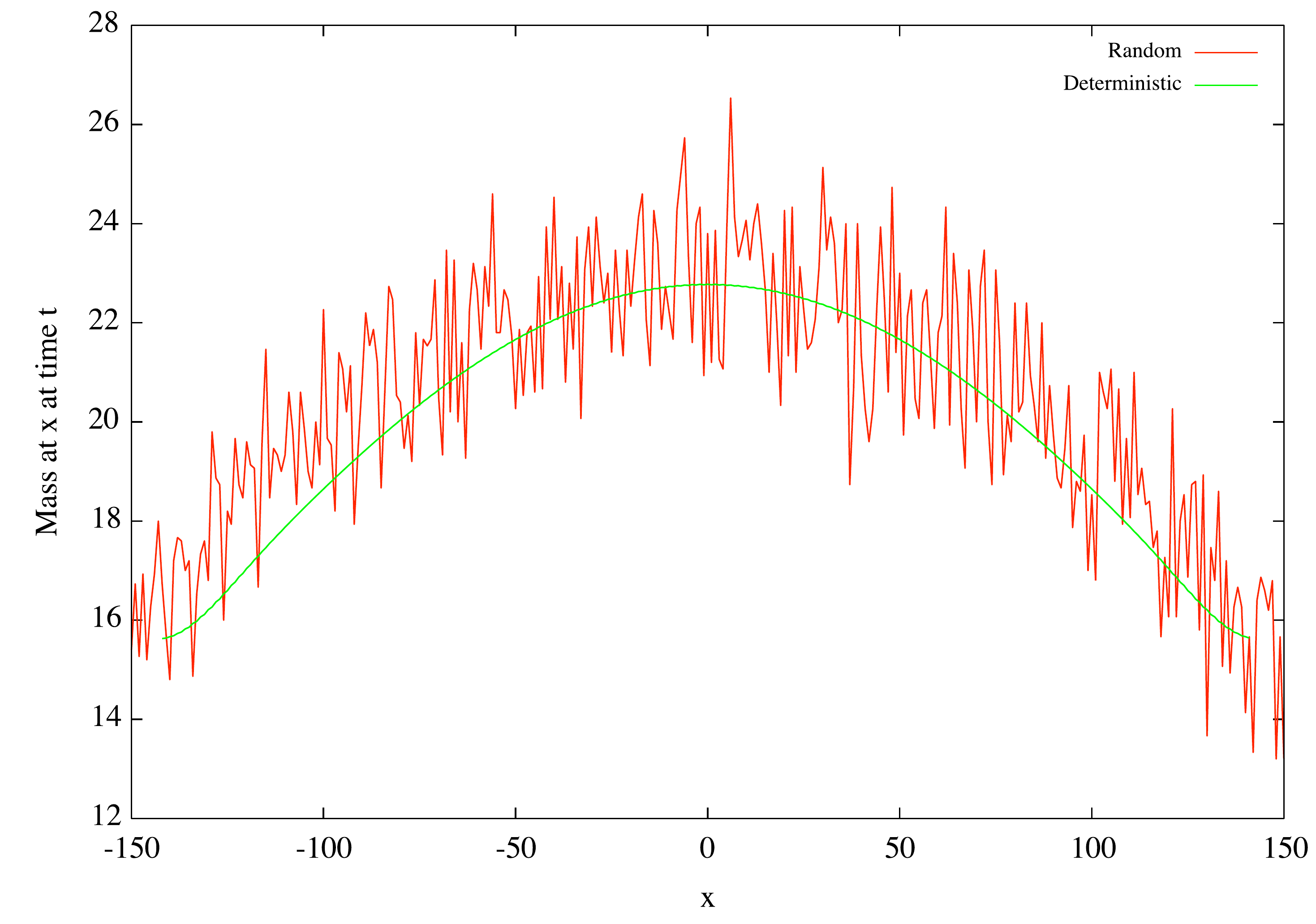}
\caption{\CD-$\frac{1}{2}$  and \textbf{Frozen Random Walk}-$(10000,\frac{1}{2})$ averaged over $15$ trials at $t=25000$.}
\label{f.mass}
\end{figure}
\end{center}

At every step $t$ of \CC-$\alpha$, we also keep track of the location of the boundary of the process, $\beta_t$, which we define as \[\beta_t:=\sup \left\{x\in \Z: \mu_t\left([x,\infty)\right) \ge \frac{\alpha}{2}\right\}.\]
We will show that
 \begin{lemma}
 \label{diffu1} 
For every $\alpha \in(0,1)$ there exist constants $a,b>0$ such that $$a\sqrt{t} < \beta_{t} < b\sqrt{t} \mbox{ }, \forall \,t. $$
\end{lemma}

The lemma above suggests that a proper scaling of $\beta_t$ is $\sqrt{t}$. Motivated by this behavior of the boundary $\beta_t$, one can ask the following natural questions:
\begin{itemize}
\item [\textbf{Q1}.]  Does $\frac{\beta_t}{\sqrt{t}}$ converge?
\end{itemize}
\noindent

Considering $\mu_t$ as a measure on $\mathbb{R}$, for $t=0,1,\ldots$ define the Borel measure $\tilde \mu_t(\alpha) = \tilde \mu_t$ on $\R$ equipped with the Borel $\sigma-$algebra such that for any Borel set $A,$ 
\begin{equation}\label{pushfow1}
\tilde \mu_{t}(A)=\mu_{t}(\{y\sqrt{t}: y\, \in\, A\}).
\end{equation}

We can now ask

\begin{itemize}
\item [\textbf{Q2}.] Does the sequence of probability measures $\tilde \mu_t$ have a weak limit?
\item [\textbf{Q3}.] If $\tilde \mu_t$ has a weak limit, what is this limiting distribution?
\end{itemize}


We conjecture affirmative answers to \textbf{Q1} and \textbf{Q2}:

\begin{conj}\label{a2} For every $\alpha \in (0,1),$ there exists $\ell_{\alpha}>0$ such that $$\displaystyle{\lim_{t \to \infty} \frac{\beta_t}{\sqrt{t}}}=\ell_{\alpha}.$$
\end{conj}

\noindent
\begin{conj}\label{a1} Fix $\alpha \in (0,1)$. Then there exists a probability measure $\mu_{\infty}(\alpha)$ on $\R$ such that as $t \to \infty,$ 
$$\tilde \mu_{t}{{\stackrel{weak}{\implies}}} \mu_{\infty}(\alpha),$$ where ${{\stackrel{weak}{\implies}}}$ denotes weak convergence in the space of finite measures on $\mathbb{R}.$ 
\end{conj}

That Conj.\ \ref{a1} implies Conj.\ \ref{a2} is the content of Lemma~\ref{strongdiff}. We now state our main result which shows that Conj.\ \ref{a2} implies Conj.\ \ref{a1} and identifies the limiting distribution, thus answering \textbf{Q3.}
To this end we need the following definition.
\begin{definition}Let $\Phi(\cdot)$ be the standard Gaussian measure on $\mathbb{R}.$ Also for any $q>0$ denote by $\Phi_{q}(\cdot),$ the probability measure on $\mathbb{R}$ which is  supported on $[-q.q]$ and  whose density is the standard Gaussian density restricted on the interval $[-q,q]$ and properly normalized to have integral $1$. 
\end{definition}

\begin{theorem}\label{main}

Assuming that $\displaystyle\lim_{t\to\infty} \frac{\beta_t}{\sqrt t}$ is a constant,  the following is true:
$$\tilde \mu_{t}{{\stackrel{weak}{\implies}}} \mu_{\infty}(\alpha),$$
where,  
$$\mu_{\infty}(\alpha)=\frac{\alpha}{2}\delta(-{q_{\alpha}})+(1-\alpha)\Phi_{q_{\alpha}}+\frac{\alpha}{2}\delta({q_{\alpha}}),$$ and $q_{\alpha}$ is the unique positive number such that:
$$\frac{\alpha}{2}q_{\alpha}=\frac{(1-\alpha)e^{-q_{\alpha}^2/2}}{\sqrt{2\pi}\Phi([-q_{\alpha},q_{\alpha}])}.$$ 
\end{theorem}

\begin{rem}\label{uniqsol} It is easy to show (see Lemma \ref{strongdiff}) that the above result implies that $$\displaystyle\lim_{t\to\infty} \frac{\beta_t}{\sqrt t}=q_{\alpha}.$$  Thus observe that by the above result,  just assuming that the boundary location properly scaled converges to a constant determines the value of the constant. This is a consequence of uniqueness of the root of a certain functional equation discussed in detail in Section \ref{pom}. 
\end{rem}
\noindent
\noindent

\subsection{Formal definitions}\label{fd1}
Let \CC-$\alpha:=\{\mu_0,\mu_1, \ldots\}$: where for each $t=0,1,\ldots,$ $\mu_t$ is a probability distribution on $\mathbb{Z}$.  For brevity we suppress the dependence on $\alpha$ in the notation since there is no scope of confusion as $\alpha$ will remain fixed throughout any argument.


\noindent 

Let $\mu_0 \equiv \delta(0)$ be the delta function at $0.$
By construction $\mu_t$ will be symmetric for all $t$.  
As described in Section \ref{def1} each $\mu_t$ contains a ``constrained/frozen" part and a ``free" part. Let the free mass and the frozen mass  be denoted by the mass distributions $\nu_t$ and $f_t$ respectively.\\
\noindent
Recall the boundary of the process, 
\begin{equation}\label{bdrydef}
\beta_t=\sup \left\{x\in \Z: \mu_t\left([x,\infty)\right) \ge \frac{\alpha}{2}\right\}.
\end{equation}
Then for all $y\ge 0,$ \begin{equation}f_t(y):=\left\{\begin{array}{cc}\mu_{t}(y) & y> \beta_{t} \\
\frac{\alpha}{2}- \displaystyle{\sum_{z> \beta_t}} \mu_t(z)& y=\beta_t\\
0 \quad & \rm {otherwise}.
\end{array}
\right.
\end{equation}
For $y<0$ let $f_{t}(y):=f_t(-y).$
Thus $f_t$ is the extreme $\alpha/2$ mass on both sides of the origin. Define the free mass to be $\nu_t:=\mu_t -f_t.$
With the above notation the heat diffusion is described by 
\begin{equation}\label{evolve}
\mu_{t+1}(x)=\frac{\nu_t(x-1)+\nu_t(x+1)}{2}+f_t(x).
\end{equation}

Recall Lemma \ref{diffu1}, which implies the diffusive nature of the boundary:
\\

\noindent
\textbf{Lemma 1.}
For every $\alpha \in(0,1)$ there exist constants $a,b>0$ such that $$a\sqrt{t} < \beta_{t} < b\sqrt{t} \mbox{ }, \forall \,t. $$

This result implies that in order to obtain any limiting statement about the measures $\mu_t$, one has to scale space down by $\sqrt{t}.$ 

\noindent

The proof of the lemma appears later. Let us first prove that the frozen mass $f_t$ cannot be supported over many points.

\begin{lemma}\label{twopoint}For all $t$, the frozen mass at time $t$, $f_t$, is supported on at most two points on each side of the origin, i.e., for all $y\in \mathbb{Z}$ such that $|y|\ge \beta_{t}+2,$ we have $f_t(y)=0.$ 
\end{lemma}
\begin{proof}

The lemma follows by induction. Assume for all $k\le t$, for all $y$ such that $|y|\ge \beta_{k}+2,$ we have $\mu_k(y)=f_k(y)=0.$ The base case $t=0$ is easy to check.
Now observe that by \eqref{evolve} and the above induction hypothesis,
\begin{equation}\label{bdrms1}
\mu_{t+1}(y)=0,
\end{equation}
for all $|y|\ge \beta_{t}+2.$
Also notice that by \eqref{evolve} it easily follows that $\beta_t$ is a non-decreasing function of $t.$
Thus clearly for all $y,$ with $|y|\ge \beta_{t+1}+2\ge  \beta_{t}+2,$
$$
\mu_{t+1}(y)=0.
$$
Hence we are done by induction.

\end{proof} 

\noindent
We now return to the proof of the diffusive nature of the boundary of the process $\beta_t$.
\\\\
\noindent
\textbf{Proof of Lemma \ref{diffu1}.} We consider the second moment of the mass distribution $\mu_t$, which we denote as $M_2(t) := \displaystyle{\sum_{x\in \mathbb{Z}}} \mu_t(x) x^2$. This is at most $(\beta_t+1)^2$ since $\mu_t$ is supported on $[-\beta_t-1, \beta_t+1]$ by Lemma \ref{twopoint}. It is also at least $\alpha{\beta_t}^2$ since there exists  mass $\alpha$ which is at a distance at least  $\beta_t$ from the origin. Now we observe how the second moment of the mass distribution evolves over time.
Suppose a free mass $m$ at $x$ splits and moves to $x-1$ and $x+1$. Then the increase in the second moment is $$\frac{m}{2}((x+1)^2+(x-1)^2)-mx^2=m.$$ Since at every time step exactly $1-\alpha$ mass is moving, the net change in the second moment at every step is $1-\alpha$. So at time $t$ the second moment is exactly 
\begin{equation}\label{sm12}
t(1-\alpha).
\end{equation}
Hence $\alpha{\beta_t}^2<t(1-\alpha)<(\beta_t+1)^2,$ and we are done. 
\qed

%
%
%

\noindent
We next prove Conjecture \ref{a2} (a stronger version of Lemma \ref{diffu1}) assuming Conjecture \ref{a1}.
\begin{lemma}\label{strongdiff} 
If Conjecture \ref{a1} holds, then so does Conjecture \ref{a2}, i.e., for every $\alpha \in (0,1),$ there exists $\ell_{\alpha}>0,$ such that $$\displaystyle{\lim_{t \to \infty} \frac{\beta_t}{\sqrt{t}}}=\ell_{\alpha}.$$
\end{lemma}


\begin{proof}Fix $\alpha \in (0,1).$  From Lemma \ref{diffu1} we know that $\{\beta_t/\sqrt{t}\}$ is bounded. Hence, if $\beta_t/\sqrt{t}$ does not converge, there exists  two subsequences $\{s_1,s_2,\ldots\}$ and $\{t_1,t_2,\ldots\}$ such that $$\lim_{i \to \infty}\beta_{s_i}/\sqrt{s_i}= q_1\mbox{ and   }\, \lim_{j \to \infty}\beta_{t_j}/\sqrt{t_j}\rightarrow q_2,$$  for some $q_2,q_1>0$ such that $q_2-q_1:=\dd>0$.
Recall $\mu_{\infty}(\alpha)$ from Conjecture \ref{a1}. Now  by hypothesis, $$\lim_{i\to \infty}\tilde \mu_{s_i} {\stackrel{weak}{\implies}} \mu_{\infty}(\alpha),\quad \lim_{j\to \infty}\tilde \mu_{t_j} {\stackrel{weak}{\implies}} \mu_{\infty}(\alpha).$$ 


This yields a contradiction since the first relation implies $\mu_{\infty}(\alpha)$ assigns mass $0$ to the interval $(q_2-\frac{\dd}{2},q_2+\frac{\dd}{2})$ while the second one implies (by Lemma \ref{twopoint}) that it assigns mass at least $\frac{\alpha}{2}$ to that interval. 
\end{proof}

\section{Proof of Theorem \ref{main}}\label{pom}
The proof follows by observing the moment evolutions of the mass distributions $\mu_{t}$ and using the moment method. The proof is split into several lemmas. Also for notational simplicity we will drop the dependence on $\alpha$ and denote $\mu_{\infty}(\alpha)$ and $q_{\alpha}$ by $\mu$ and $q$ respectively since $\alpha $ will stay fixed in any argument. 
Thus \begin{equation}\label{cand123}
 \mu:=\mu_{\infty}(\alpha)=\frac{\alpha}{2}\delta(-{q})+(1-\alpha)\Phi_{q}+\frac{\alpha}{2}\delta({q}).
\end{equation}
Also denote the $k^{th}$ moments of $\mu_t$ as $M_{k}(t)$. We now make some simple observations which are consequences of the previously stated lemmas.
Recall the free and frozen mass distributions $\nu_{t}$ and $f_{t}.$ We denote the $k^{th}$ moments of the measures $\nu_t$ (the free mass at time $t$), $f_t$ (the frozen mass at time $t$), by $M^{\nu}_k(t)$ and $M^{f}_{k}(t)$ respectively. Also define $\tilde f_t$ and $\tilde \nu_t$ similarly to $\tilde \mu_t$ in \eqref{pushfow1}.
\noindent
Assuming Conjecture \ref{a2}, it follows from Lemma \ref{twopoint} that,
\begin{equation}\label{frozconv}
\tilde f_{t} {\stackrel{weak}{\implies}}f 
\end{equation}
where $f:=\frac{\alpha}{2}\delta(-\ell_{\alpha})+\frac{\alpha}{2}\delta(\ell_{\alpha}),$ and $\ell_{\alpha}$ appears in the statement of Conj \ref{a2}. 
This implies that \begin{align}\label{mcon1}
\frac{M^{f}_{k}(t)}{t^{k/2}}=\left\{ \begin{array}{cc}
0,& k \mbox{ odd} \\
\alpha \ell_{\alpha}^{k}(1+o(1)), & k \mbox{ even}
\end{array}
\right.
\end{align}
where $o(1)$ goes to $0$ as $t$ goes to infinity.\\\\
\noindent 
The proof of  Theorem \ref{main} is in two steps: first we show that $\ell_{\alpha}=q$  and then show that  $\tilde \nu_{t}$ converges weakly to the  part of $\mu$ which is absolutely continuous with respect to the Lebesgue measure. Clearly the above two results combined imply Theorem \ref{main}.
\\

\noindent
As mentioned this is done by observing the moment sequence $M_{k}(t).$ Now notice owing to symmetry of the measures $\mu_t$ for any $t$, $M_{2k+1}(t)=0$ for all non-negative integers $k.$\\
Thus it suffices  to consider $M_{2k}(t)$ for some non-negative integer $k$. We begin by observing that at any time $t$ the change in the moment $M_{2k}(t+1) - M_{2k}(t),$ is caused by the movement of the free mass $\nu_t$. The change caused by a mass $m$ moving at a site $x$ (already argued in the proof of Lemma \ref{diffu1} for $k=1$) is 
\begin{equation}\label{com1}
\frac{m((x+1)^{2k}+(x-1)^{2k})}{2}-mx^{2k}= m\left[\sum_{i=1}^{k} {2k \choose 2k-2i} x^{2k-2i}      \right].
\end{equation}
\noindent
Now summing over $x$ we get that,
\begin{equation} 
M_{2k}(t+1) - M_{2k}(t)=  \sum_{i=1}^{k}   {2k \choose 2k-2i} M^{\nu}_{2k-2i}(t).
\label{eq:momentchange}
\end{equation}
Notice that the moments of the free mass distribution $\nu_t$ appear on the RHS since $m$ in \eqref{com1} was the free mass at a site $x$. Now using \eqref{eq:momentchange} we sum $M_{2k}(j+1) - M_{2k}(j)$  over $0\le j \le t-1 $  and normalize by $t^{k}$  to get
\begin{equation} 
\label{evolut1}
\frac{M_{2k}(t)}{t^{k}} = \sum_{j=0}^{t-1} \left[ \sum_{i=1}^{k} {2k \choose 2k-2i} M^{\nu}_{2k-2i}(j) \frac{1}{t^k}   \right].
\end{equation}
Recall that by Lemma \ref{diffu1}, for any $k\ge 1,$ $M_{2k-2}^{\nu}(j)$ is $O(j^{k-1})$.
Moreover, the above equation allows us to make the following observation: \\
\noindent 
\textbf{Claim.} Assume \eqref{mcon1} holds. Then for any $k\ge 1$, the existence of $\displaystyle\lim_{j \to \infty} \frac{M_{2k-2}^{\nu}(j)}{j^{k-1}}$  implies existence of
$\displaystyle\lim_{j \to \infty} \frac{M_{2k}^{\nu}(j)}{j^{k}}$.\\\\
\noindent
\textit{Proof of claim.} Notice that by Lemma \ref{diffu1}, $M_{2k-\ell}^{\nu}(j) =O(j^{k-2})$ for any $\ell \le 4$.
Also let $$\lim_{j\to \infty}\frac{M_{2k-2}^{\nu}(j)}{j^{k-1}}=M^{\nu}_{2k-2},$$ which exists by hypothesis.
  
Thus using \eqref{evolut1} and the standard fact that $$\lim_{t \to \infty}\sum_{j=0}^{t-1}  \frac{j^{k-1}}{t^{k-1}} \frac{1}{t}= \int_{0}^{1} x^{k-1} dx = \frac{1}{k}$$ we get 
$$\sum_{j=1}^{t-1} \left[ \sum_{i=1}^{k} {2k \choose 2k-2i} \frac{M^{\nu}_{2k-2i}(j)  }{t^k}  \right]
=(2k-1)M^{\nu}_{2k-2}+o(1)+O\left(\frac{1}{t}\right).$$ 
Thus 
\begin{equation}
\label{lim}
\displaystyle\lim_{t\to\infty} \frac{M_{2k}(t)}{t^{k}} = (2k-1)M^{\nu}_{2k-2}
\end{equation}
\noindent
and since 
\begin{equation}\label{decomp1}
M_{2k}(t)=M^{\nu}_{2k}(t)+M^{f}_{2k}(t),
\end{equation}
we are done by \eqref{mcon1}.

\qed
 \\\\

\noindent
Using the above claim, the fact that $\displaystyle\lim_{t\to \infty }\frac{M_{k}(t)}{t^{k/2}}$ and hence, $\displaystyle\lim_{t\to \infty }\frac{M^{\nu}_{k}(t)}{t^{k/2}}$ (by  \eqref{decomp1} and \eqref{mcon1}) exists for all $k$, follows from the fact that $\frac{M_2(t)}{t}=(1-\alpha)$ (see \eqref{sm12}). Let us call the limits $M_{k}$ and $M^{\nu}_{k}$ respectively.
\\\\
\noindent
Thus we have \begin{equation}\label{eq:momenteq}
M_{2k}=M^{\nu}_{2k}+\alpha \ell_{\alpha}^{2k}=(2k-1)M^{\nu}_{2k-2},
\end{equation}
where the first equality is by \eqref{decomp1} and \eqref{mcon1} and the second by \eqref{lim}.
For $k=1$ we get  $${\alpha \ell_{\alpha}^2}+M^{\nu}_2=1-\alpha.$$
Notice that this implies that for all $k$, $M^{\nu}_{2k}$ can be expressed in terms of a polynomial in $\ell_{\alpha}$ of degree $2k$, which we denote as $P_{k}(\ell_{\alpha})$. Then, by \eqref{eq:momenteq} the polynomials $P_k$ satisfy the following recurrence relation:
\begin{eqnarray}\label{eq:recursion}
P_k(\ell_{\alpha}) &=&(2k-1)P_{k-1}(\ell_{\alpha})-\alpha \ell_{\alpha}^{2k}\\
\nonumber
P_{0}& =& 1-\alpha.
\end{eqnarray} 
By definition, we have
\begin{equation}\label{conv12}
P_k(\ell_{\alpha})=M^{\nu}_{2k}=\lim_{t\to \infty} \frac{M^{\nu}_{2k}(t)}{t^{k}}=\lim_{t\to \infty} \frac{\displaystyle{\sum_{-\beta_t\le x\le \beta_t}}x^{2k}\nu_{t}(x)}{t^{k}}.
\end{equation} 
Thus assuming Conj. \ref{a2} and the fact that $\displaystyle{\sum_{-\beta_t\le x\le \beta_t}}\nu_{t}(x) =1-\alpha$ for all $t$, we get the following family of inequalities,
\begin{equation}\label{eq:ineq}
0 \le P_k(\ell_{\alpha})\le (1-\alpha)\ell_{\alpha}^{2k}\,\mbox{   }\forall\,k\ge0.
\end{equation}
We next show that the above inequalities are true only if $\ell_{\alpha}=q$ where $q$ appears in \eqref{cand123}.
\begin{lemma}\label{thmq}
The inequalities in (\ref{eq:ineq}) are satisfied by the unique number $\ell_{\alpha}$ such that 
$$\frac{\alpha}{2}\ell_{\alpha}=\frac{(1-\alpha)e^{-\ell_{\alpha}^2/2}}{\sqrt{2\pi}\Phi([-\ell_{\alpha},\ell_{\alpha}])}$$ where $\Phi(\cdot)$ is the standard Gaussian measure.
\end{lemma}
Thus the above implies that necessarily $\ell_{\alpha}=q$ where $q$ appears in \eqref{cand123}. This was mentioned in Remark \ref{uniqsol}. 
\begin{proof}
To prove this, first we write the inequalities in a different form so that the polynomials stabilize. 
To this goal, let us define $$\tilde{P}_{k}=\frac{P_{k}}{(2k-1)!!}$$ where $(2k-1)!!=(2k-1)(2k-3)\ldots 1.$  
Then it follows from \eqref{eq:recursion} that $$\tilde{P}_{k}(\ell_{\alpha})=\tilde{P}_{k-1}(\ell_{\alpha})-\frac{\alpha}{(2k-1)!!}\ell_{\alpha}^{2k}.$$
Hence $$\tilde{P}_k(\ell_{\alpha})=\left(1-\alpha-\sum_{i=1}^{k}\frac{\alpha \ell_{\alpha}^{2i}}{(2i-1)!!}\right).$$
The inequalities in \eqref{eq:ineq} translate to 
\begin{equation}\label{eq:powerineq}
0\le 1-\alpha-\sum_{i=1}^{k}\frac{\alpha \ell_{\alpha}^{2i}}{(2i-1)!!}\le \frac{\ell_{\alpha}^{2k+2}}{(2k+1)!!}.
\end{equation}
Let us first identify the power series $$g(x)=\sum_{i=1}^{\infty}\frac{x^{2i-1}}{(2i-1)!!}.$$
Clearly the power series converges absolutely for all values of $x$. It is also standard to show that one can interchange differentiation and the sum in the expression for $g(\cdot)$. Thus we get that,
$$\frac{dg(x)}{dx}=1+xg(x).$$ Solving this differential equation using integrating factor $e^{-x^2/2}$ and the fact that $g(0)=0$ we get $$g(x)=e^{x^2/2}\int_{0}^{x}{e^{-y^2/2}}dy.$$ 
As $k \to \infty $, the upper bound in \eqref{eq:powerineq} converges to $0$ for any value of $\ell_{\alpha}$. Also the expression in the middle converges to $1-\alpha-\alpha \ell_{\alpha}g(\ell_{\alpha}).$
Thus taking the limit in \eqref{eq:powerineq} as $k\rightarrow \infty$ we get that $\ell_{\alpha}>0$ satisfies
\begin{equation}\label{key}
\ell_{\alpha}g(\ell_{\alpha})=\frac{1-\alpha}{\alpha}.
\end{equation}
Clearly this is the same as the equation appearing in the statement of the lemma. Also notice that since $x g(x)$ is monotone on the positive real axis, by the uniqueness of the solution of \eqref{key} we get $\ell_{\alpha}=q$ where $q$ appears in \eqref{cand123}. Hence we are done.
\end{proof}
\noindent

The value of $\ell_{\alpha}$ that solves \eqref{key} when $\alpha=1/2$ is approximately $0.878$. Figure \ref{f.con} shows the numerical convergence of $\beta_t/\sqrt{t}$ to $q_{\alpha}$ for various values of $\alpha$.

\begin{figure}[H]
\begin{center}
\centering
\includegraphics[scale=1]{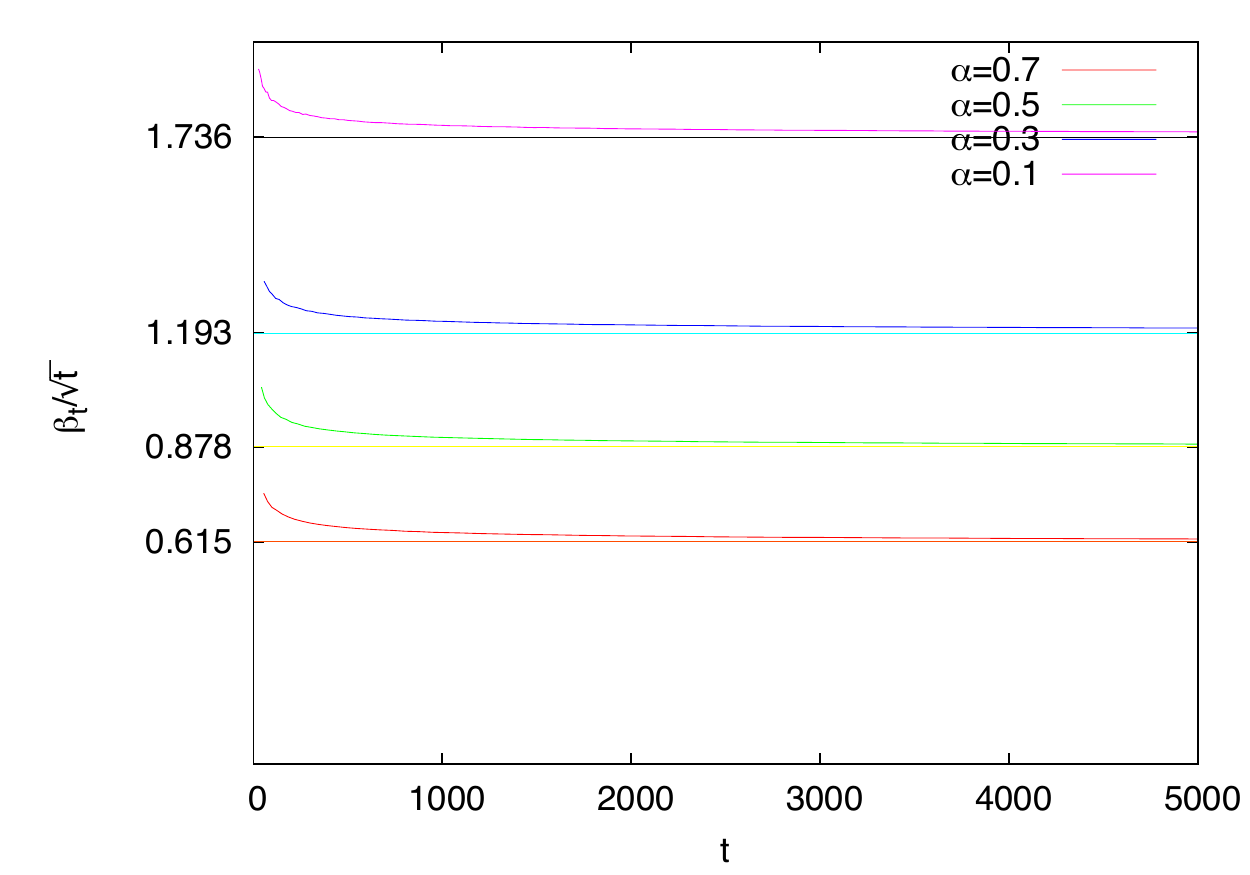}
\caption{Convergence of $\beta_t/\sqrt{t}$ for various $\alpha$. The horizontal lines denote the values $q_{\alpha}$ and the curves plot $\frac{\beta_t}{\sqrt{t}}$ as a function of time $t$. }
\label{f.con}
\end{center}
\end{figure}

Thus, assuming Conjecture \ref{a2}, by Lemma \ref{thmq}, $\tilde f_t$ converges to $f$ (as stated in \eqref{frozconv}) which consists of two atoms of size $\frac{\alpha}{2}$ at $q$ and $-q$. 
To conclude the proof of Theorem \ref{main}, we now show $\tilde \nu_{t}$ converges to the absolutely continuous part of $\mu$ (see \eqref{cand123}). 
Recall that by \eqref{conv12} and Lemma \ref{thmq} the $2k^{th}$ moment of $\tilde \nu_t$ converges to $P_{k}(q).$
We will use the following well known result:

\begin{lemma}[30.1, \cite{billingsley}]
\label{moments}
Let $\mu$ be a probability measure on the line having finite moments $\alpha_k = \int_{-\infty}^{\infty} x^k \mu(dx)$ of all orders. If the power series $\displaystyle\sum_k \alpha_k r^k/k!$ has a positive radius of convergence, then $\mu$ is the only probability measure with the moments $\alpha_1, \alpha_2, \ldots$.
\end{lemma}


Thus, to complete the proof of Theorem \ref{main} we need to show the following:
\\

\textbf{Claim.} The $2k^{th}$ moment of the measure  $(1-\alpha)\Phi_{q}$ is $P_{k}(q)$ where $q$ appears in \eqref{cand123}.
\\




To prove this claim, it suffices to show that the moments of $(1-\alpha)\Phi_{q}$ satisfy the recursion \eqref{eq:recursion}. Recall that $q=q_{\alpha}$. Let $C=C_{\alpha}:=\frac{\sqrt{2\pi}\Phi([-q,q])}{1-\alpha}.$  
Using integration by parts we have:
\begin{eqnarray*}
 \frac{\int_{-q}^{q} x^{2k}e^{-\frac{x^2}{2}} dx} {C} &=& \frac{ \int_{-q}^{q}x^{2k-1}xe^{-\frac{x^2}{2}} dx }{C} \\
&=& -\frac{2q^{2k-1}e^{-\frac{q^2}{2}}}{C}+\frac{(2k-1)}{C}\int_{-q}^{q}x^{2k-2}e^{-\frac{x^2}{2}}dx.
\end{eqnarray*}

By the relation that $q$ satisfies in the statement of Theorem \ref{main}, the first term on the RHS without the $-$ sign is $\alpha q^{2k}$. Also, note that the second term is $(2k-1)$ times the $(2k-2)^{nd}$ moment of $(1-\alpha)\Phi_q$. Thus, the moments of $(1-\alpha)\Phi_{q}$ satisfy the same recursion as in \eqref{eq:recursion}.
\\

Now from Example 30.1 in \cite{billingsley}, we know that the absolute value of the $k^{th}$ moment of the standard normal distribution is bounded by $k!$. Then, similarly, the absolute value of the $k^{th}$ moment of our truncated Gaussian, $\Phi_q$, is bounded by $c^k k!$ for a constant $c$. Then Lemma~\ref{moments} implies that $\Phi_q$ is determined by its moments and quoting Theorem 30.2 in \cite{billingsley} we are done.

\qed

%
%

\section{Concluding Remarks}\label{conc} 
We conclude with a brief discussion about a possible approach towards proving Conjectures \ref{a2} and \ref{a1} and some experiments in higher dimensions. \\

\begin{figure}[H]
\centering
\includegraphics[scale=0.7]{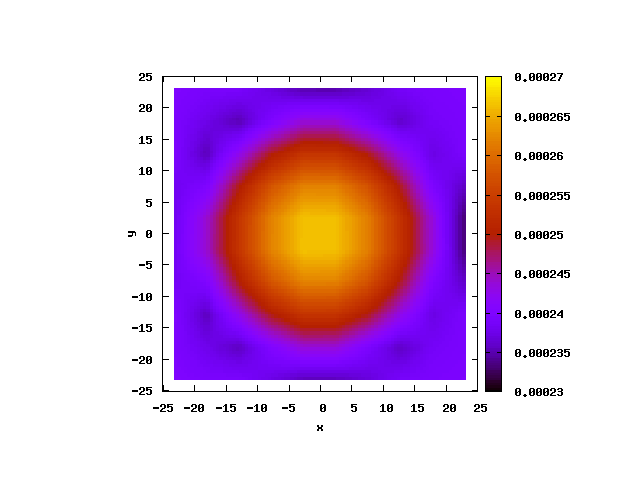}
\caption{Heat map of the free mass distribution after 1000 steps in $2$ dimensions for \CC-1/2.}
\label{hm1}
\end{figure}
The free part $\nu_t$ of the distribution $\mu_t$ could represent the distribution of a random walk in a growing interval. If the interval boundaries grow diffusively, the scaling limit of this random process will be a reflected Ornstein-Uhlenbeck process on this interval $[-q,q]$. 
We remark that the stationary measure for Ornstein-Uhlenbeck process reflected on the interval  is known to be the same truncated Gaussian which appears in Theorem \ref{main}, see \cite[(31)]{ouh}. This connection could be useful in proving the conjectures.

We also note that similar results are expected in higher dimensions; in particular, the mass distribution should exhibit rotational symmetry. See Fig \ref{hm1}. Note the truncated Gaussian shape for the slices $x=0$ (Fig \ref{sl1}).

\begin{figure}[H]
\centering
\includegraphics[scale=0.4]{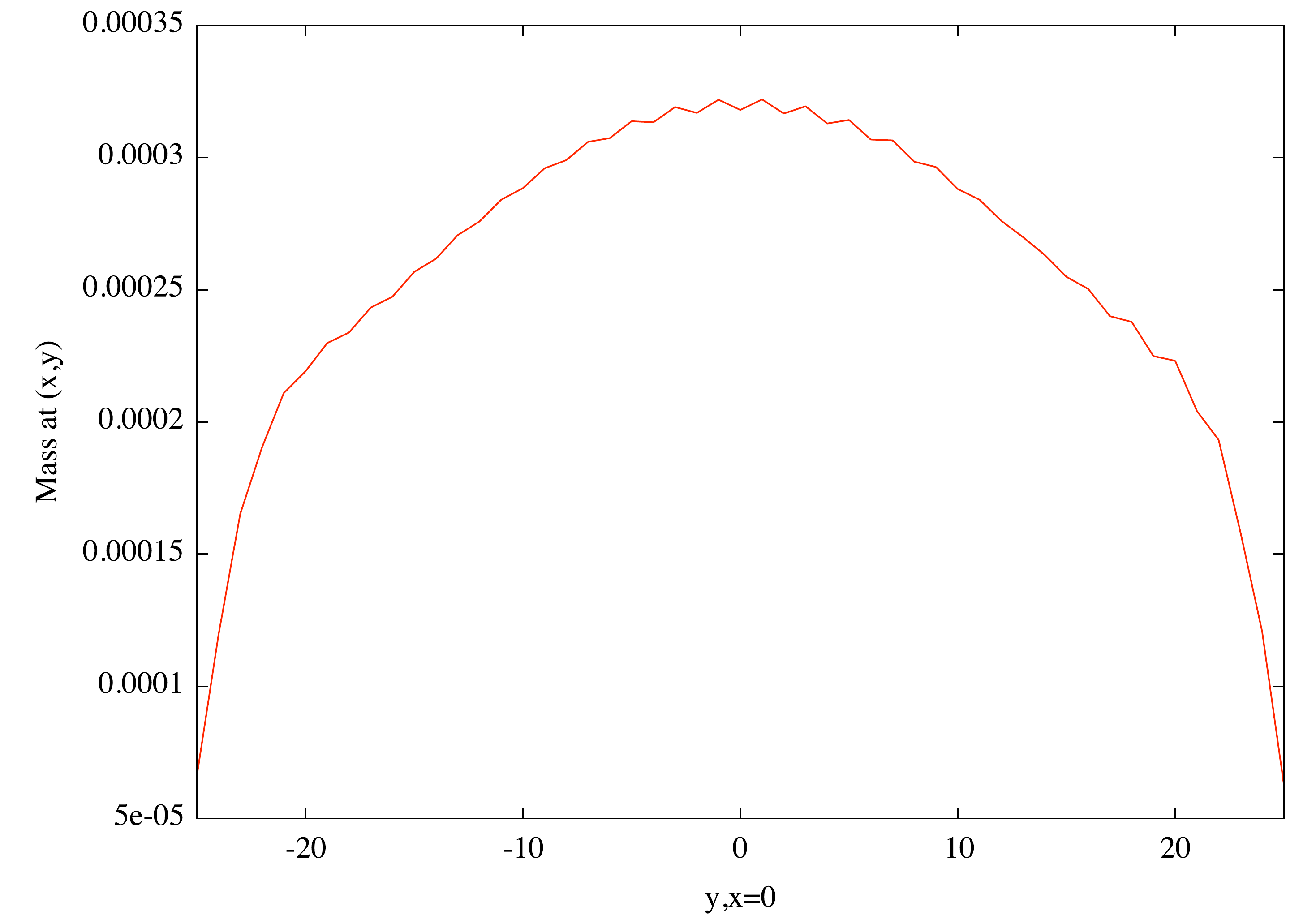}
\caption{Slice of the free mass distribution at $x=0$ after 1000 steps in $2$ dimensions for the analogue of \CC-1/2.}
\label{sl1}
\end{figure}

\section*{Acknowledgements}
The authors thank  Matan Harel, Arjun Krishnan and Edwin Perkins for helpful discussions. Part of this work has been done at Microsoft Research in Redmond and the first two authors thank the group for its hospitality.

\bibliographystyle{plain}
  \bibliography{freeze}

\end{document}